\documentclass[12pt]{article}
\usepackage{amsmath, xcolor}
\usepackage{amsmath, amsthm, amscd, amsfonts, amssymb, graphicx}

\setbox0=\hbox{$+$}
\newdimen\plusheight
\plusheight=\ht0
\def\+{\;\lower\plusheight\hbox{$+$}\;}

\setbox0=\hbox{$-$}
\newdimen\minusheight
\minusheight=\ht0
\def\-{\;\lower\minusheight\hbox{$-$}\;}

\setbox0=\hbox{$\cdots$}
\newdimen\cdotsheight
\cdotsheight=\plusheight
\def\cds{\lower\cdotsheight\hbox{$\cdots$}}

\renewcommand{\(}{\left\(}
\renewcommand{\)}{\right\)}

\renewcommand{\pmod}[1]{\,(\textup{mod}\,#1)}
\numberwithin{equation}{section}
\theoremstyle{plain}
\newtheorem{theorem}{Theorem}[section]
\newtheorem{lemma}[theorem]{Lemma}

\newtheorem{corollary}[theorem]{Corollary}

\begin{document}

\vspace{-2cm}

\begin{center}
	
	\begin{center}{\Large\textbf{Some new results for Andrews’ Kimberling partitions}}\end{center}\vskip5mm
	{\bf Gaurab Bardhan$^1$ and Nipen Saikia$^{2, \ast}$}\\
	$^1$Department of Mathematics, Tyagbir Hem Baruah College,\\ Jamugurihat, Sonitpur, Assam, India.\\
	E. Mail: gaurabbardhan561@gmail.com
	\vskip2mm
	$^2$Department of Mathematics, Rajiv Gandhi
	University,\\ Rono Hills, Doimukh, Arunachal Pradesh, India.\\
	E. Mail(s): nipennak@yahoo.com\\
	$^\ast$\textit{Corresponding author}.\end{center}
\noindent{\bf Abstract.}
  George E. Andrews (2016) introduced the Kimberling index, $K(\pi)$, of a partition $\pi$ of a positive integer $n$, which is defined  as
  $$ K(\pi) = (\text{largest part of } \pi) - (\text{least part of } \pi) - (\text{number of parts of } \pi). $$  
   Based on Kimberling index, Andrews defined five partition functions, $K_>(n),$ $K_<(n),$ $K_\leq(n),$ $K_=(n),$ and $K_\geq(n)$, called Kimberling partition functions, which count the numbers of partitions of a positive integer $n$ for which the Kimberling index $K(\pi) $ is $>0$, $<0$, $\leq0$, $=0$ and $\geq 0$, respectively. He also gave the generating functions  for $K_\le(n),$ $K_<(n),$ and $K_>(n)$ and established some relations connecting Kimberling partitions and other partition functions. Since then, the Kimberling partition functions  and their generating functions remained unexplored. In this paper, we derive  generating functions for $K_=(n),$ and $K_\geq(n)$,  and  establish some congruence relations of the five Kimberling partition functions by using the method of $q$-series identities. \\

 \vskip2mm


\section{Introduction}
For any positive integer $n$ and any complex numbers $\alpha$ and $q$ with $|q|<1$, define  $q$-series in standard notations as
\begin{equation}\label{ie1}
   (\alpha;q)_{0}=1,\quad  (\alpha;q)_{n}=\prod_{j=0}^{n-1}(1-\alpha q^{j}),\quad(\alpha;q)_{\infty}=\prod_{j=0}^{\infty}(1-\alpha q^{j}). 
\end{equation}Throughout this paper, we will use the notation $g_j:=(q^j;q^j)_\infty$ for any positive integer $j$.

The non-increasing finite sequence of positive integers $\alpha_1\geq \alpha_2 \geq \alpha_3 \geq \cdots\geq \alpha_k>0$ is said to be a partition of a positive integer $n$ if $n=\sum_{i=1}^{k}\alpha_i$. The integers $\alpha_i$ are called parts or summands of the partition. For example, the partitions of $n=4$ are given by $4,\quad 3+1, \quad 2+2,\quad 2+1+1,\quad 1+1+1+1.$ If $p(n)$ denotes the number of partitions of a positive integer $n$ then $p(4)=5$.
Euler \cite{euler1748introductio} gave the generating function of $p(n)$ as
\begin{equation}\label{e1}
  \sum_{n=0}^{\infty} p(n) q^n = \dfrac{1}{(q; q)_\infty}; \qquad p(0)=1.  
\end{equation} 
Using the $q$-binomial theorem \cite{1}
\begin{equation}\label{ie3}
  \sum_{n=0}^{\infty}\dfrac{(\alpha;q)_{n}}{(q;q)_{n}}z^{n}=\dfrac{(\alpha;q)_{\infty}}{(q;q)_{\infty}},
\end{equation}
\eqref{e1} can be written as 
\begin{equation}\label{ier}
 \sum_{n=0}^{\infty} p(n) q^n= \sum_{n=0}^{\infty}\dfrac{q^n}{(q;q)_n}.
\end{equation}

Andrews \cite{kimberling} introduced the Kimberling index, $K(\pi),$ of a partition $\pi$ of a positive integer $n$ as
\begin{equation}
 K(\pi) = (\text{largest part of } \pi) - (\text{least part of } \pi) - (\text{number of parts of } \pi).   
\end{equation}
For example, the Kimberling index of the partition $\pi: 11 + 7 + 4 + 3$ of $n=25$  is  $11 - 3 - 5 = 3$. Andrews \cite{kimberling}  further defined five partition functions $K_>(n),$ $K_<(n),$ $K_\leq(n),$ $K_=(n)$ and $K_\geq(n)$, called Kimberling partition functions, which counts the numbers of partitions of a positive integer $n$ for which the Kimberling index $K(\pi)$ is $>0$, $<0$, $\leq0$, $=0$ and $\geq 0$, respectively. The five Kimberling partition functions corresponds to  the sequences A237803-A237834 in the OEIS \cite{ck}.

Andrews  \cite{kimberling} gave following generating functions for the  Kmberling partition functions $K_{\leq}(n), K_<(n)$ and $K_{>}(n)$, respectively: 
\begin{align}
	\sum_{n= 1}^{\infty} K_{\leq}(n) q^n &=\label{k<=1}\sum_{n= 1}^{\infty} \frac{q^n (q^{n+1}; q)_{n-1}}{(q; q)_n}\\
	&= \label{k<=2}\frac{1}{(q; q)_\infty} \sum_{n=1}^\infty (-1)^{n-1} n q^{n(3n-1)/2} (1 + q^n),\end{align} \begin{align} 
	\sum_{n= 1}^{\infty}K_<(n)q^n &=\label{k<1} \sum_{n= 1}^{\infty} \frac{q^n (q^n; q)_{n-1}}{(q; q)_n}\\
	&=\label{k<2} \frac{1}{(q; q)_\infty (1 - q)} \sum_{n=0}^\infty (-1)^{n-1} q^{3n(n-1)/2+1} (1 - q^{2n}),\\
	\sum_{n= 1}^{\infty} K_{>}(n) q^n &=\label{k>1} \sum_{n= 1}^{\infty}\frac{q^n (1 - (q^{n+1}; q)_{n-1})}{(q; q)_n}\\
	&\label{k>2} =\frac{1}{(q; q)_\infty} \sum_{n=1}^\infty (-1)^n (n-1) q^{n(3n-1)/2} (1 + q^n),
\end{align} 
and  proved following relatons connecting the Kimberling partitions,
\begin{equation}\label{k=}
    K_<(n) + K_=(n) + K_>(n) = p(n),
\end{equation}
\begin{equation}\label{k<}
    K_<(n) + K_=(n) = K_\leq(n),
\end{equation}
\begin{equation}\label{k}
    K_>(n) + K_=(n) = K_\geq(n).
\end{equation} He also established some relations connecting Kimberling partitions with some  other partition functions.
From \eqref{k=}-\eqref{k}, it is clear that, any two of the five Kimberling partition functions together with the partition functions $p(n)$ is sufficient to determine the other three Kimberling partition functions. Since then, the Kimberling partition functions  and their generating functions remained unexplored. In this paper, we derive  generating functions for $K_=(n),$ and $K_\geq(n)$,  and  establish some congruence relations of the five Kimberling partition functions by connecting them with some other arithmetic functions  by using the method of $q$-series identities.

Following Andrews, to  derive generating  functions for $K_=(n)$ and $K_\geq(n)$ we use the  Gaussian polynomials of $q$-binomial coefficients and the Fine's notation. The Gaussian polynomials of $q$-binomial coefficients is defined as follows:
\begin{equation}
    \begin{bmatrix} N \\ M \end{bmatrix} = \begin{cases}
 \dfrac{(q;q)_N }{(q;q)_M(q;q)_{N-M}},\quad  0 \leq M \leq N \\
 0,  \quad  \quad  \quad  \quad  \quad  \quad  \quad  \quad  \text{otherwise}.
\end{cases}   
\end{equation}
From \cite [p.33, Theorem 3.1]{1}, we note that the polynomial $\begin{bmatrix} N+M \\ M \end{bmatrix}$
is the generating function for partitions in which each part is $\leq N$ and the number of parts is $\leq M$.\\
The Fine's notation \cite[p.53, (25.94)]{fine} is defined as,
\begin{equation}\label{Q1}
    Q(a;q) := \sum_{m= 0}^{\infty} \frac{(aq^{m+1};q)_m q^m}{(q;q)_m}=\frac{1}{(q;q)_\infty} \sum_{m= 0}^{\infty} (-a)^m q^{3m(m+1)/2}.
\end{equation}
 From \eqref{k=},  we note that
    \begin{equation}\label{T2e1}
     \sum_{n=1}^{\infty}K_{=}(n)q^n=\sum_{n=1}^{\infty}p(n)q^n-\sum_{n=1}^{\infty}K_{<}(n)q^n-\sum_{n=1}^{\infty}K_{>}(n)q^n.
\end{equation}
Employing \eqref{ier}, \eqref{k<1} and \eqref{k>1} in \eqref{T2e1}, we obtain 
\begin{equation}\label{T2e2}
     \sum_{n=1}^{\infty}K_{=}(n)q^n=\sum_{n=1}^{\infty}\dfrac{q^n}{(q;q)_n}-\sum_{n=1}^{\infty}\dfrac{q^n(q^{n};q)_{n-1}}{(q;q)_n}-\sum_{n=1}^{\infty}\dfrac{q^n\left(1-(q^{n+1};q)_{n-1}\right)}{(q;q)_n}
\end{equation}
which on simplification gives
\begin{equation}\label{T2e3}
     \sum_{n=1}^{\infty}K_{=}(n)q^n=\sum_{n=1}^{\infty}\dfrac{q^{2n-1}(q^{n+1};q)_{n-2}\left(q-q^{n}\right)}{(q;q)_n}.
\end{equation}
Similarly, from \eqref{T2e3}, we obtain
 \begin{equation}\label{T2E1}
     \sum_{n=1}^{\infty}K_{=}(n)q^n=\sum_{n=1}^{\infty}\dfrac{q^{2n}(q^{n+1};q)_{n-2}\left(1-q^{n-1}\right)}{(q;q)_n}.
\end{equation}
Again, employing \eqref{ier} and \eqref{k<1} in \eqref{k=}, we obtain
\begin{equation}\label{k>=1}
\sum_{n=1}^{\infty}K_{\geq}(n)q^n= \sum_{n=1}^{\infty}\frac{q^n (1 - (q^{n}; q)_{n-1})}{(q; q)_n}
\end{equation}

In \cite{hop}, Hopkins and Sellers defined Garden of Eden partitions with those partitions having rank less than $-1$, where rank of a partition is the defined as the largest part  of partition minus number of parts.  If $ge(n)$ denotes the number of Garden of Eden partitions of $n$ then its generating function \cite{hop} is given by
\begin{align}
     \sum_{n=0}^\infty ge(n)q^n &=\label{gen} \frac{1}{(q; q)_\infty} \sum_{n=1}^\infty (-1)^{n-1} q^{3n(n+1)/2}\\
     &=\frac{1}{(q; q)_\infty} \left( q^3 - q^9 + q^{18} - q^{30} + \cdots \right).
\end{align}

An overpartition $\pi$ of a positive integer $n$ can be defined as a partition of $n$, where the first occurrence of each part can be overlined. Let $s(\pi)$ denote the smallest part of the partition $\pi$ of a positive integer $n$ and  $F(n)$ denote the number of overpartitions $\pi$ of $n$ into distinct parts where $s(\pi)$ occurs overlined and the non-overlined parts are in the half-open interval $[s(\pi), 2s(\pi))$. If $F_0(n)$ (resp. $F_1(n)$) denote the number of overpartition counted by $F(n)$ in which the number of parts is even (resp. odd), then  Andrews and Bachraoui \cite{F} established that
\begin{align}
   \sum_{n=1}^\infty F'(n)q^n=& \sum_{n=1}^\infty (F_1(n) - F_0(n))q^n\\
   =&\label{F}\sum_{n=1}^\infty q^n(q^{n+1}; q)_\infty(q^n; q)_n.
\end{align}

\section{Preliminaries}
Ramanujan's general theta function $f(\alpha, \beta)$ \cite[p. 34, (18.1)]{BBC} is defined by
\begin{equation}\label{gtheta}
   f(\alpha, \beta) = \sum_{m=-\infty}^{\infty} \alpha^{m(m+1)/2} \beta^{m(m-1)/2}, \qquad |\alpha\beta| < 1.
\end{equation}
Three important cases for $f(\alpha, \beta)$ are the functions $\phi(q), \psi(q)$, and $f(-q)$ \cite[p. 35, Entry 18]{BBC} which are defined as
\begin{equation}\label{phi}
    \phi(q) := f(q, q) = \sum_{m=-\infty}^{\infty} q^{m^2} = \dfrac{g_2^5}{g_1^2g_4^2},
\end{equation}
\begin{equation}\label{psi}
    \psi(q) := f(q, q^3) = \sum_{m=0}^{\infty} q^{m(m+1)/2} = \dfrac{g_2^2}{g_1},
\end{equation}
and 
\begin{equation}\label{f}
    f(-q) := f(-q, -q^2) = \sum_{m=-\infty}^{\infty} (-1)^m q^{m(3m-1)/2} = g_1.
\end{equation}
From \eqref{f}, it is easily seen that
\begin{equation}\label{eptt}
    q\dfrac{d}{dq}g_1=q\dfrac{d}{dq}(q;q)_\infty= \sum_{m=-\infty}^{\infty} (-1)^m\dfrac{m(3m-1)}{2} q^{m(3m-1)/2}. 
\end{equation}
The product representations of the  special cases \eqref{phi}-\eqref{f} of $f(\alpha, \beta)$ are the consequences of the Jacobi's triple product \cite[ p. 35, Entry 19]{BBC} given by
$$f(\alpha, \beta) = (-\alpha; \alpha\beta)_\infty(-\beta; \alpha\beta)_\infty(\alpha\beta; \alpha\beta)_\infty.$$
The Jacobi triple product \cite[Theorem 1.2.1]{ramurty} can also be stated as
\begin{equation}\label{jtp1}
\sum_{n=-\infty}^\infty z^n q^{n^2}=\prod_{m=1}^{\infty}(1-q^{2m})(1+zq^{2m-1})(1+z^{-1}q^{2m-1}).
\end{equation}
The following Jacobi's identity \cite[Theorem 1.3.9]{bc} will also be useful,
\begin{equation}\label{jtp}
g_{1}^{3} = \sum_{n=0}^{\infty}(-1)^{n}(2n+1)q^{n(n+1)/2}. 
\end{equation}
Further recall that, the Lambert series for the number of divisors of $n$,  $d(n)$, is given by 
\begin{equation}\label{d(n)}
    \sum_{n=1}^{\infty}d(n)q^n=\sum_{n=1}^{\infty}\dfrac{q^n}{1-q^n}
\end{equation}
and Lambert series for the sum of divisors of $n$, $\sigma(n)$, is given by 
	\begin{equation}\label{sum}
	\sum_{n=1}^\infty \dfrac{n q^{n}}{1 - q^{n}}=\sum_{n=1}^\infty \sigma(n)q^n.	
    \end{equation}
Therefore, if $d_{k,j}(n),$ denotes the number of positive divisors of $n$ such that $d\equiv j\pmod k$, then 
\begin{equation}\label{djk(n)}
    \sum_{n=1}^{\infty}d_{k,j}(n)q^n=\sum_{n=0}^{\infty}\dfrac{q^{kn+j}}{1-q^{kn+j}}.
\end{equation}
Also by taking logarithim on both sides of \eqref{e1}, then differentiating both sides with respect to $q$ and simplifying using \eqref{sum}, we obtain
\begin{equation}\label{sig1}
    -\dfrac{q\dfrac{d}{dq}(q;q)_\infty}{(q;q)_\infty}=\sum_{n=1}^\infty \sigma(n)q^n.
\end{equation}

In addition to above,following dissection formulas will be used in proving the congruences:
\begin{lemma}\cite{bk}\label{dsctn}
We have
\begin{equation}\label{d271}
g_1^4 = \frac{g_4^{10}}{g_2^2 g_8^4} - 4q \frac{g_2^2 g_8^4}{g_4^2}.
\end{equation}
\begin{equation}\label{d2}
\frac{g_1^2}{g_2} = \frac{g_9^2}{g_{18}} - 2q \frac{g_3 g_{18}^2}{g_6 g_9}.
\end{equation}
\end{lemma}

\begin{lemma}\label{RR5}
\cite{m5h} We have
\begin{align}
  \frac{1}{(q; q)_\infty} = &\frac{(q^{25}; q^{25})_\infty^5}{(q^5; q^5)_\infty^6} (F^{-4}(q^5) + qF^{-3}(q^5)+ 2q^2 F^{-2}(q^5) + 3q^3 F^{-1}(q^5) + 5q^4 - 3q^5 F(q^5)\notag\\ 
  &\label{RR5e}+ 2q^6 F^2(q^5) - q^7 F^3(q^5) + q^8 F^4(q^5)), 
\end{align}
\end{lemma}
where $F(q) := q^{-1/5}R(q) $ and $R(q) $ is Rogers–Ramanujan continued fraction defined by
$$
R(q) := q^{1/5}\frac{(q;q^5)_\infty(q^4;q^5)_\infty}{(q^2;q^5)_\infty(q^3;q^5)_\infty}=\dfrac{q^{1/5}}{1}_+\dfrac{q}{1}_+\dfrac{q^2}{1}_+ \dfrac{q^3}{1}_{+\cdots} , \quad |q| < 1.
$$

Next two lemmas follow from the Jacobi's identity \eqref{jtp}.
\begin{lemma}\label{m7}
\cite[Lemma 2.1]{mh7} We have
\begin{equation}\label{m7E1}
  g_1^6\equiv \left(J_0(q^7) + qJ_1(q^7) + q^3J_3(q^7)\right)^2 \pmod{7},
\end{equation}
where $ J_i, i = 0, 1, 3 $, re power series with integral powers of $q^7$.
\end{lemma}

\begin{lemma}\cite{m11h} \label{m11}
We have
\begin{equation}\label{l11E1}
  g_1^3\equiv \mathcal{J}_0(q^{11})+q\mathcal{J}_1(q^{11})+q^3\mathcal{J}_3(q^{11})+q^6\mathcal{J}_6(q^{11})+q^{10}\mathcal{J}_{10}(q^{11}) \pmod{11},
\end{equation}
where $\mathcal{J}_i,$ $i = 0,$ $1,$ $3,$ $6,$ $10$ are power series with integral powers of $q^{11}$.
\end{lemma}
\begin{lemma}
    \cite[Corollary, p.49]{BBC} We have 
    \begin{align}
\phi(q)&=\label{newpsi1}\phi(q^9)+2qf(q^3,q^{15})\\
&=\label{newpsi2}\phi(q^{25})+2qf(q^{15},q^{35})+2q^4f(q^5,q^{45})
        \end{align}
\end{lemma}
Following lemma follows from \eqref{ie1} and the binomial theorem: 
\begin{lemma}\label{modp}
	For any prime $p$, we have
	$$g_p\equiv g_1^p\pmod{p}.$$
	$$g_{p}^{p}\equiv g_1^{p^2}\pmod{p^2}.$$
\end{lemma}

\section{Congruences for the Kimberling partition functions}
Recall that, for $k\in\mathbb{N}\cup \lbrace0\rbrace$ and $m\in\mathbb{Z},$ pentagonal numbers $P_m$ and triangular numbers $T_k$ are respectively defined by  $P_m=\dfrac{m(3m-1)}{2}$ and $T_k=\dfrac{k(k+1)}{2}$. Throughout this section, we will use the symbol $\lceil x\rceil $ and $\lfloor x\rfloor$ for the ceiling and floor function respectively (that is, smallest integer greater than or equal to $x$ and the smallest integer less than or equal to $x$ respectively).
\begin{theorem}For any integers $n$, $k>0,$ we have
    \begin{equation}\label{ept1e1}\hspace{-4cm}(i)~\sum_{m=\lfloor\frac{1-\sqrt{1+24n}}{6}\rfloor+1}^{\lceil\frac{1+\sqrt{1+24n}}{6}\rceil-1}(-1)^m K_{\leq}(n-P_m)=\begin{cases}
 k(-1)^{k-1},\quad \text{if $n=P_{\pm k}$ } \\
 0,  \hspace{1.8cm}\text{otherwise}.
\end{cases}   
    \end{equation}
    \begin{align}\hspace{-.38cm}(ii)~\sum_{m=\lfloor\frac{1-\sqrt{1+24(n+1)}}{6}\rfloor+1}^{\lceil\frac{1+\sqrt{1+24(n+1)}}{6}\rceil-1}(-1)^m K_{<}(n-P_m+1)-&\sum_{m=\lfloor\frac{1-\sqrt{1+24n}}{6}\rfloor+1}^{\lceil\frac{1+\sqrt{1+24n}}{6}\rceil-1}(-1)^m K_{<}(n-P_m)\notag\\
     &\label{ept1e2}=\begin{cases}
 (-1)^{k},\quad \text{if $n=P_{-k},$ or $n=3T_{k}$ } \\
 0,  \hspace{1.2cm}  \text{otherwise}.
\end{cases}   
    \end{align}
    \begin{equation}\label{ept1e3}\hspace{-3cm}(iii)~\sum_{m=\lfloor\frac{1-\sqrt{1+24n}}{6}\rfloor+1}^{\lceil\frac{1+\sqrt{1+24n}}{6}\rceil-1}(-1)^m K_{>}(n-P_m)=\begin{cases}
 (k-1)(-1)^{k},\quad \text{if $n=P_{\pm k}$ } \\
 0,  \hspace{2.3cm}  \text{otherwise}.
\end{cases}   
    \end{equation}

\end{theorem}
\begin{proof}(i)  \eqref{k<=2} can be written as
        \begin{equation}{\label{epT1e2}}
         (q;q)_\infty \sum_{n=1}^\infty K_{\leq}(n) q^n=\sum_{n=1}^\infty (-1)^{n-1} n q^{n(3n-1)/2} (1 + q^n).
        \end{equation} Employing \eqref{f} in \eqref{epT1e2}, we obtain
        \begin{equation}{\label{epT1e3}}
         \left(\sum_{m=-\infty}^{\infty} (-1)^m q^{m(3m-1)/2}\right)\left( \sum_{n=1}^\infty K_{\leq}(n) q^n\right)=\sum_{n=1}^\infty (-1)^{n-1} n q^{n(3n-1)/2} (1 + q^n),
        \end{equation}
        which on employing the definition of $P_m$ and taking $n>P_m$ gives
         \begin{equation}{\label{epT1e4}}
         \sum_{n=1}^{\infty}\left(\sum_{m=\lfloor\frac{1-\sqrt{1+24n}}{6}\rfloor+1}^{\lceil\frac{1+\sqrt{1+24n}}{6}\rceil-1}(-1)^m K_{\leq}(n-P_m)\right)q^n=\sum_{n=1}^\infty (-1)^{n-1} n q^{n(3n-1)/2} (1 + q^n).
        \end{equation}
 Comparing the coefficients of like powers of $q$ in both sides of \eqref{epT1e4}, we complete the proof.
    
\noindent (ii) \eqref{k<2} can be written as
\begin{equation}\label{T5e1}
\sum_{n=1}^\infty K_{<}(n) q^n =\frac{q}{(1-q)(q; q)_\infty} \left(\sum_{n=1}^\infty (-1)^{n}q^{3n(n+1)/2}+\sum_{n=0}^\infty (-1)^{n}q^{n(3n+1)/2}\right)
\end{equation}which is equivalent to
\begin{equation}\label{T12e1}
\sum_{n=1}^\infty K_{<}(n) q^{n-1}-\sum_{n=1}^\infty K_{<}(n) q^{n}  =\frac{1}{(q; q)_\infty} \left(\sum_{n=1}^\infty (-1)^{n}q^{3n(n+1)/2}+\sum_{n=0}^\infty (-1)^{n}q^{n(3n+1)/2}\right).
\end{equation}Employing \eqref{f}, the definition of $P_m$ and taking $n>P_m$ in \eqref{T12e1}, we obtain
$$
\hspace{-1cm}\sum_{n=1}^\infty\left(\sum_{m=\lfloor\frac{1-\sqrt{1+24(n+1)}}{6}\rfloor+1}^{\lceil\frac{1+\sqrt{1+24(n+1)}}{6}\rceil-1}(-1)^m K_{<}(n-P_m+1)-\sum_{m=\lfloor\frac{1-\sqrt{1+24n}}{6}\rfloor+1}^{\lceil\frac{1+\sqrt{1+24n}}{6}\rceil-1}(-1)^m K_{<}(n-P_m)\right)q^n$$
\begin{equation}\hspace{2cm}=\sum_{n=1}^\infty (-1)^{n}q^{3n(n+1)/2}+\sum_{n=0}^\infty (-1)^{n}q^{n(3n+1)/2}.
\end{equation} Comparing the coefficients of like powers of $q$ in both sides of \eqref{epT1e4}, we complete the proof of  (ii). 
Proof of (iii) is identitcal to the proof of (i) and follows from \eqref{k>2}.

\end{proof}

\begin{theorem}For any positive integer $n>1,$  we have
    $$\hspace{-3.5cm}(i)~
         K_{\leq}(n)=d_{3,1}(n)-d_{3,2}(n)+2\left(\sum_{m=1}^{\lfloor\frac{-1+\sqrt{1+24n}}{6}\rfloor}m(-1)^{m-1} p(n-P_{-m})\right).
    $$
   $$\hspace{-9.3cm}(ii)~
       K_{\leq}(n)\equiv d_{3,1}(n)-d_{3,2}(n)\pmod 2.
    $$
\end{theorem}
\begin{proof}
Let \begin{equation}\label{t1e2}
    f(z)=\sum_{n=-\infty}^{\infty}(-1)^nz^nq^{n(3n-1)/2}.
\end{equation}
Differentiating \eqref{t1e2} with respect to $z$ and then substituting $z=1$ in it, we obtain
\begin{equation}\label{t1e3}
    f'(1)=\sum_{n=-\infty}^{\infty}(-1)^nnq^{n(3n-1)/2}.
\end{equation}
Substituting $q$ by $q^{1/2}$ and $z$ by $-zq^{1/2}$ in \eqref{jtp1}, we obtain 
\begin{equation}\label{t1e4}
    \sum_{n=-\infty}^{\infty}(-1)^nz^nq^{n(n+1)/2}=(zq;q)_\infty(1/z;q)_\infty(q;q)_\infty.
\end{equation}
Again, replacing $z$ by $z/q$ in \eqref{t1e4}, we obtain
\begin{equation}\label{t1e5}
    \sum_{n=-\infty}^{\infty}(-1)^nz^nq^{n(n-1)/2}=(z;q)_\infty(q/z;q)_\infty(q;q)_\infty.
\end{equation}
Using \eqref{t1e2} in \eqref{t1e5}, we obtain
\begin{equation}\label{t1ae7}
    f(z)=\sum_{n=-\infty}^{\infty}(-1)^n(zq)^n(q^3)^{n(n-1)/2}
    =(zq;q^3)_\infty(q^2/z;q^3)_\infty(q^3;q^3)_\infty.
\end{equation}
Taking logarithim on both sides of \eqref{t1ae7} and then diffrentiating it with respect to $z$ on both sides, we obtain 
\begin{equation}\label{t1e8}
    \dfrac{f'(z)}{f(z)}=-\sum_{k=0}^{\infty}\dfrac{q^{3k+1}}{1-zq^{3k+1}}+\sum_{k=0}^{\infty}\dfrac{q^{3k+2}}{z^2-zq^{3k+2}}.
\end{equation}
Putting $z=1$ in \eqref{t1e8} and then employing \eqref{f}, we obtain 
\begin{equation}\label{t1e9}
    -\dfrac{f'(1)}{(q;q)_\infty}=\sum_{k=0}^{\infty}\dfrac{q^{3k+1}}{1-q^{3k+1}}-\sum_{k=0}^{\infty}\dfrac{q^{3k+2}}{1-q^{3k+2}}.
\end{equation}
Employing \eqref{djk(n)}, \eqref{t1e3}, and \eqref{t1e9} in \eqref{k<=2}, we obtain
\begin{align}
    \sum_{n \geq 1} K_{\leq}(n) q^n=&\label{t1e11}\sum_{n=0}^{\infty}\left(\dfrac{q^{3n+1}}{1-q^{3n+1}}-\dfrac{q^{3n+2}}{1-q^{3n+2}}\right)+\dfrac{2}{(q;q)_\infty}\sum_{n=1}^{\infty}(-1)^{n-1}nq^{n(3n+1)/2}\\
    =&\sum_{n=0}^{\infty}\left(d_{3,1}(n)-d_{3,2}(n)\right)q^n+2\left(\sum_{n=0}^{\infty}p(n)q^n\right)\left(\sum_{n=1}^{\infty}(-1)^{n-1}nq^{n(3n+1)/2}\right)\\
=&\label{t1e10}\sum_{n=0}^{\infty}\left(d_{3,1}(n)-d_{3,2}(n)\right)q^n+2\sum_{n=2}^{\infty}\left(\sum_{m=1}^{\lfloor\frac{-1+\sqrt{1+24n}}{6}\rfloor}m(-1)^{m-1}p(n-P_{-m})\right)q^n.
\end{align}
Comparing the coefficients of the like powers of $q$ on both sides of \eqref{t1e10}, we  arrive at (i). (ii) follows immediately from (i).
\end{proof}

\begin{theorem} \label{thm1}We have\\
$(i)$~For $n>2,$ 
\begin{align}
  \sum_{m=1}^{n-1}K_{\leq}(n-m)\sigma(m)\equiv&\sum_{m=1}^{\lceil\frac{1+\sqrt{1+24n}}{6}\rceil-1}(-1)^{m}m P_mp(n-P_m)\notag\\
  &+\sum_{m=1}^{\lceil\frac{-1+\sqrt{1+24n}}{6}\rceil-1}(-1)^{m}mP_{-m}p(n-P_{-m})\pmod n. \notag
\end{align}
$(ii)$~For $n>3,$ 
\begin{align}
\hspace{.5cm} &\sum_{m=1}^{n-1}K_{<}(n-m)\sigma(m)-\sum_{m=1}^{n}K_{<}(n-m+1)\sigma(m)\notag\\
  &+\sum_{m= 1}^{\lceil\frac{-3+\sqrt{9+24n}}{6}\rceil-1}(-1)^{m-1}3T_mp(n-3T_m)+\sum_{m= 0}^{\lfloor\frac{-1+\sqrt{1+24n}}{6}\rfloor}(-1)^{m-1}P_{-m}p(n-P_{-m})\equiv0\pmod n.\notag
  \end{align}
$(iii)$~For $n>2,$ 
\begin{align}
  \sum_{m=1}^{n-1}K_{>}(n-m)\sigma(m)\equiv&\sum_{m=1}^{\lceil\frac{1+\sqrt{1+24n}}{6}\rceil-1}(-1)^{m-1}(m-1) P_mp(n-P_m)\notag\\
  &\label{T33E1}+\sum_{m=1}^{\lceil\frac{-1+\sqrt{1+24n}}{6}\rceil-1}(-1)^{m-1}(m-1)P_{-m}p(n-P_{-m})\pmod n. \notag
\end{align}
\end{theorem}
\begin{proof}
(i)~ Let, $$G_1(q)=\sum_{n=1}^\infty K_{\leq}(n) q^n\quad \text{and }\quad P_1(q)=\sum_{n=1}^\infty (-1)^{n-1} n q^{n(3n-1)/2} (1 + q^n).$$
Taking logarithim on both sides of \eqref{k<=2}, we obtain
\begin{equation}\label{T3e1}
\text{log}\left(G_1(q)\right)=\text{log}\left(P_1(q)\right)-\sum_{m=1}^{\infty}\text{log}\left(1-q^m\right).
\end{equation}
Now differentiating \eqref{T3e1}  with respect to $q$ and then multiplying by $q$, we obtain 
\begin{equation}\label{T3e2}
    q\dfrac{G'_1(q)}{G_1(q)}=q\dfrac{P'_1(q)}{P_1(q)}+\sum_{m=1}^{\infty}\dfrac{mq^{m}}{\left(1-q^m\right)}.
\end{equation}
Employing \eqref{sum} in \eqref{T3e2} and then multipling both sides by $G_1(q)$, we obtain
\begin{align}
    \sum_{n \geq 1}n K_{\leq}(n) q^n&=\dfrac{1}{(q;q)_\infty}\sum_{m=1}^\infty (-1)^{m-1}\dfrac{m^2\left(3m-1\right)}{2} q^{m(3m-1)/2}\notag \\
    &\label{T3e3}\hspace{-1cm}+\dfrac{1}{(q;q)_\infty}\sum_{m=1}^\infty (-1)^{m-1}\dfrac{m^2\left(3m+1\right)}{2} q^{m(3m+1)/2}+\sum_{n>1}\left(\sum_{m=1}^{n-1}K_{\leq}(n-m)\sigma(m)\right)q^n.
\end{align}
Employing \eqref{e1} in \eqref{T3e3}, then multiplying the infinite series in the right hand side and comparing the coefficients of like powers of $q$ on both sides, we obtain that for $n>2$,
\begin{align}
   \hspace{-.5cm}nK_{\leq}(n)=&\sum_{m=1}^{n-1}K_{\leq}(n-m)\sigma(m)+ \sum_{m=1}^{\lceil\frac{1+\sqrt{1+24n}}{6}\rceil-1}(-1)^{m-1}m P_mp(n-P_m)\notag\\
   &\label{T3e4}\hspace{5cm}+\sum_{m=1}^{\lceil\frac{-1+\sqrt{1+24n}}{6}\rceil-1}(-1)^{m-1}mP_{-m}p(n-P_{-m}).
\end{align} 
Now (i) follows easily from \eqref{T3e4}.\\
(ii)~ Let, $$G_2(q)=\sum_{n=1}^\infty K_{<}(n) q^n \quad \text{and }\quad
P_2(q)=\sum_{n=1}^\infty (-1)^{n}q^{3n(n+1)/2}+\sum_{n=0}^\infty (-1)^{n}q^{n(3n+1)/2}.$$
Then by \eqref{k<2}, we obtain
\begin{equation}\label{T32e1}
   G_2(q)=\dfrac{q}{(1-q)(q; q)_\infty}P_2(q). 
\end{equation}
Taking logarithim on both sides of \eqref{T32e1} and then differentiating both sides with respect to $q,$ we obtain
\begin{equation}\label{T32e2}
   \dfrac{G'_2(q)}{G_2(q)}=\dfrac{1}{q}+\dfrac{P'_2(q)}{P_2(q)}+\dfrac{1}{1-q}+\sum_{m=1}^{\infty}\dfrac{mq^{m-1}}{\left(1-q^m\right)}. 
\end{equation}
Multiplying both sides of \eqref{T32e2} by $q$ and then simplifying by employing \eqref{e1}, \eqref{sum}, we obtain
\begin{align}
\sum_{n=1}^\infty n K_{<}(n) q^n=&\sum_{n=1}^\infty K_{<}(n) q^n
+\dfrac{q}{(1-q)}\sum_{n=1}^\infty K_{<}(n) q^n
+\sum_{n=2}^\infty\left(\sum_{m=1}^{n-1}K_{<}(n-m)\sigma(m)\right)q^n\notag\\
&\hspace{-.5cm}\label{T32e3}+\dfrac{q}{(1-q)}\left(\sum_{n=0}^\infty p(n)q^n\right)\left(\sum_{m=1}^\infty(-1)^m3T_mq^{3T_m}+\sum_{m=0}^{\infty}(-1)^mP_{-m}q^{P_{-m}}\right)\notag\\ 
\end{align}
Comparing the like powers of $q$ in both sides of \eqref{T32e3}, we deduce that,  for $n>3$
$$
n\left( K_{<}(n+1)-n K_{<}(n)\right)=
+\sum_{m=1}^{n}K_{<}(n-m+1)\sigma(m)-\sum_{m=1}^{n-1}K_{<}(n-m)\sigma(m)$$
\begin{equation}\label{T32e4}\sum_{m= 1}^{\lceil\frac{-3+\sqrt{9+24n}}{6}\rceil-1}(-1)^m3T_mp(n-3T_m)+\sum_{m= 0}^{\lfloor\frac{-1+\sqrt{1+24n}}{6}\rfloor}(-1)^mP_{-m}p(n-P_{-m}).
\end{equation}
Now (ii) folows immediately from \eqref{T32e4}. \\(iii) can be proved using  \eqref{k>2} and following  similar steps as in the proof of (i).
\end{proof}

\begin{corollary}
    For $n>2,$ and $m\in\mathbb{Z}$, we have
    $$
 \hspace{-5.5cm}(i)~\sum_{m=1}^{n-1}\left(K_{>}(n-m)+K_{\leq}(n-m)\right)\sigma(m)\equiv-\sigma(n)\pmod n,\notag\\
$$
$$\hspace{-8.9cm}(ii)~\sum_{m=1}^{n-1}p(n-m)\sigma(m)\equiv-\sigma(n)\pmod n,$$
\end{corollary}

\begin{proof} From Theorem \ref{thm1}(i) and (iii), we note that
    $$
        \hspace{-2.3cm}\sum_{m=1}^{n-1}\left(K_{>}(n-m)+K_{\leq}(n-m)\right)\sigma(m)\equiv
        \sum_{m=1}^{\lceil\frac{1+\sqrt{1+24n}}{6}\rceil-1}(-1)^{m}P_mp(n-P_m)
  $$\begin{equation}\label{recr}\hspace{7cm}+\sum_{m=1}^{\lceil\frac{-1+\sqrt{1+24n}}{6}\rceil-1}(-1)^{m}P_{-m}p(n-P_{-m})\pmod n \end{equation}
On the other hand, using \eqref{e1} and \eqref{eptt}, we obtain
\begin{equation}\label{recra1}
\left(\sum_{n=0}^{\infty}p(n)q^n\right)
\left(
\sum_{m=-\infty}^{\infty}
(-1)^mP_mq^{P_m}
\right)=
\dfrac{1}{(q;q)_\infty}
\left(
q\dfrac{d}{dq}(q;q)_\infty
\right).
\end{equation}
Employing \eqref{sig1} in \eqref{recra1}, we obtain
\begin{equation}\label{recrr3}
\left(\sum_{n=0}^{\infty}p(n)q^n\right)
\left(
\sum_{m=-\infty}^{\infty}
(-1)^mP_mq^{P_m}
\right)
=
-\sum_{n=1}^{\infty}\sigma(n)q^n.
\end{equation}
Comparing the coefficients of like powers of $q$ on both sides of
\eqref{recrr3}, we obtain
\begin{equation}\label{recr2}
\sum_{m=1}^{\left\lfloor
\frac{1+\sqrt{1+24n}}{6}\right\rfloor}
(-1)^mP_m p(n-P_m)+\sum_{m=1}^{\left\lfloor
\frac{-1+\sqrt{1+24n}}{6}\right\rfloor}
(-1)^mP_{-m}p(n-P_{-m})
=-\sigma(n).
\end{equation}
Employing \eqref{recr2} in \eqref{recr}, we obtain
$$\sum_{m=1}^{n-1}
\left(K_{>}(n-m)+K_{\leq}(n-m)\right)\sigma(m)
\equiv-\sigma(n)\pmod n,$$
which proves $(i)$. Finally, from \eqref{k=} and \eqref{k<}, we have
$$K_{>}(n)+K_{\leq}(n)=p(n).$$
Thus, $(ii)$ follows immediately from $(i)$.
\end{proof}

\begin{theorem}
For $n>0,$  we have
\begin{align}
\hspace{-1cm}(i)&~
\sum_{m=\lfloor\frac{1-\sqrt{1+12n}}{6}\rfloor+1}
^{\lceil\frac{1+\sqrt{1+12n}}{6}\rceil-1}
\sum_{k=0}^{\lceil\frac{-1+\sqrt{1+8n}}{2}\rceil-1}
(-1)^mK_{\leq}(n-T_k-2P_m)\notag\\
&\hspace{1cm}\equiv
\sum_{\substack{r,s,t\in\mathbb Z,\ l\geq1\\
n=P_r+P_s+2P_t+P_l}}
(-1)^{r+s+t+l-1}l
+\sum_{\substack{r,s,t\in\mathbb Z,\ l\geq1\\
n=P_r+P_s+2P_t+P_{-l}}}
(-1)^{r+s+t+l-1}l
\pmod4\\
\hspace{-1cm}(ii)&~
\sum_{m=\lfloor\frac{1-\sqrt{1+12n}}{6}\rfloor+1}
^{\lceil\frac{1+\sqrt{1+12n}}{6}\rceil-1}
\sum_{k=0}^{\lceil\frac{-1+\sqrt{1+8n}}{2}\rceil-1}
(-1)^mK_{<}(n-T_k-2P_m)\notag\\
&\hspace{1cm}\equiv
\sum_{\substack{r,s,t\in\mathbb Z,\ j\geq1,\ l\geq0\\
n=P_r+P_s+2P_t+j+P_{-l}}}
(-1)^{r+s+t+l}
+\sum_{\substack{r,s,t\in\mathbb Z,\ j\geq1,\ l\geq1\\
n=P_r+P_s+2P_t+j+3T_l}}
(-1)^{r+s+t+l}
\pmod4,\\
\hspace{-1cm}(iii)&~
\sum_{m=\lfloor\frac{1-\sqrt{1+12n}}{6}\rfloor+1}
^{\lceil\frac{1+\sqrt{1+12n}}{6}\rceil-1}
\sum_{k=0}^{\lceil\frac{-1+\sqrt{1+8n}}{2}\rceil-1}
(-1)^mK_{>}(n-T_k-2P_m)\notag\\
&\hspace{1cm}\equiv
\sum_{\substack{r,s,t\in\mathbb Z,\ l\geq1\\
n=P_r+P_s+2P_t+P_l}}
(-1)^{r+s+t+l}(l-1)
+\sum_{\substack{r,s,t\in\mathbb Z,\ l\geq1\\
n=P_r+P_s+2P_t+P_{-l}}}
(-1)^{r+s+t+l}(l-1)
\pmod4.
\end{align}
\end{theorem}
\begin{proof}
Let,
$$P_1(q)=
\sum_{n=1}^{\infty}
(-1)^{n-1}nq^{\frac{n(3n-1)}{2}}(1+q^n).$$
Then \eqref{k<=2} can be written as
\begin{equation}\label{T4e1}
\sum_{n=1}^{\infty}K_{\leq}(n)q^n
=
\dfrac{1}{g_1}P_1(q).
\end{equation}
Multiplying \eqref{T4e1} by $\left(g_2^3/g_1\right)$, we obtain
\begin{equation}\label{T4e3}
\dfrac{g_2^3}{g_1}
\sum_{n=1}^{\infty}K_{\leq}(n)q^n
=
\dfrac{g_2^3}{g_1^2}P_1(q).
\end{equation}
Employing Lemma \ref{modp} in \eqref{T4e3}, we obtain
\begin{equation}\label{T4e4}
\dfrac{g_2^3}{g_1}
\sum_{n=1}^{\infty}K_{\leq}(n)q^n
\equiv g_1^2g_2P_1(q)\pmod4.
\end{equation}
Employing \eqref{psi} in \eqref{T4e4}, we obtain
\begin{equation}\label{T4e6}
g_2\psi(q)
\sum_{n=1}^{\infty}K_{\leq}(n)q^n
\equiv g_1^2g_2P_1(q)\pmod4.
\end{equation}
Employing \eqref{f}, we obtain
\begin{align}
g_1^2g_2
&\label{t2a2}=
\left(
\sum_{r=-\infty}^{\infty}(-1)^rq^{P_r}
\right)
\left(
\sum_{s=-\infty}^{\infty}(-1)^sq^{P_s}
\right)
\left(
\sum_{t=-\infty}^{\infty}(-1)^tq^{2P_t}
\right).
\end{align}
Employing \eqref{psi} and \eqref{t2a2} in \eqref{T4e6}, we
obtain
\begin{align}
&\left(
\sum_{m=-\infty}^{\infty}(-1)^mq^{2P_m}
\right)
\left(
\sum_{k=0}^{\infty}q^{T_k}
\right)
\left(
\sum_{n=1}^{\infty}K_{\leq}(n)q^n
\right)\notag\\
&\label{T4e7}\equiv
\left(
\sum_{r=-\infty}^{\infty}(-1)^rq^{P_r}
\right)
\left(
\sum_{s=-\infty}^{\infty}(-1)^sq^{P_s}
\right)
\left(
\sum_{t=-\infty}^{\infty}(-1)^tq^{2P_t}
\right)
\left(
\sum_{l=1}^{\infty}
(-1)^{l-1}l
\left(q^{P_l}+q^{P_{-l}}\right)
\right)
\pmod4.
\end{align}
Applying the Cauchy product on the left-hand side of \eqref{T4e7},
we obtain
\begin{align}
&\sum_{n=1}^{\infty}
\left(
\sum_{m=\lfloor\frac{1-\sqrt{1+12n}}{6}\rfloor+1}
^{\lceil\frac{1+\sqrt{1+12n}}{6}\rceil-1}
\sum_{k=0}^{\lceil\frac{-1+\sqrt{1+8n}}{2}\rceil-1}
(-1)^mK_{\leq}(n-T_k-2P_m)
\right)q^n\notag\\
&\label{T4e8}\equiv
\left(
\sum_{r=-\infty}^{\infty}(-1)^rq^{P_r}
\right)
\left(
\sum_{s=-\infty}^{\infty}(-1)^sq^{P_s}
\right)
\left(
\sum_{t=-\infty}^{\infty}(-1)^tq^{2P_t}
\right)
\left(
\sum_{l=1}^{\infty}
(-1)^{l-1}l
\left(q^{P_l}+q^{P_{-l}}\right)
\right)
\pmod4.
\end{align}
Comparing the coefficients of like powers of $q$ in
\eqref{T4e8}, we obtain
\begin{align}
&\sum_{m=\lfloor\frac{1-\sqrt{1+12n}}{6}\rfloor+1}
^{\lceil\frac{1+\sqrt{1+12n}}{6}\rceil-1}
\sum_{k=0}^{\lceil\frac{-1+\sqrt{1+8n}}{2}\rceil-1}
(-1)^mK_{\leq}(n-T_k-2P_m)\notag\\
&\hspace{1cm}\equiv
\sum_{\substack{r,s,t\in\mathbb Z,\ l\geq1\\
n=P_r+P_s+2P_t+P_l}}
(-1)^{r+s+t+l-1}l
+\sum_{\substack{r,s,t\in\mathbb Z,\ l\geq1\\
n=P_r+P_s+2P_t+P_{-l}}}
(-1)^{r+s+t+l-1}l
\pmod4.
\label{T4e9}
\end{align}
This proves $(i)$.\\\\
Proof of (ii)  and (iii) are similar to the proof of (i) and follow from the  generating functions of $K_{<}(n),$ and $K_{>}(n)$, respectively.  
\end{proof}
\begin{theorem}\label{T5}
    For $n>1,$ we have 
$$\hspace{-2cm}\sum_{m=\lfloor\frac{1-\sqrt{1+24(n-1)}}{2}\rfloor+1}^{\lfloor\frac{1+\sqrt{1+24(n-1)}}{2}\rfloor}(-1)^mK_{<}(n-P_m-1)-\sum_{m=\lfloor\frac{1-\sqrt{1+24n}}{2}\rfloor+1}^{\lfloor\frac{1+\sqrt{1+24n}}{2}\rfloor}(-1)^mK_{<}\left(n-P_m\right)$$
$$\hspace{4.5cm}=\sum_{m=\lceil\frac{1-\sqrt{1+24(n-1)}}{2}\rceil}^{\lfloor\frac{1+\sqrt{1+24(n-1)}}{2}\rfloor}(-1)^{m+1}ge(n-P_m-1)+\sum_{m=0}^{n-2}p(m)F'(n-m-1).  
$$\end{theorem}
\begin{proof}

Employing \eqref{Q1} and \eqref{gen} in \eqref{T5e1}, we obtain
\begin{align}
\sum_{n=1}^\infty K_{<}(n) q^n &=\frac{q}{(1-q)} \left(-\sum_{n=0}^\infty ge(n)q^n+\frac{1}{(q; q)_\infty}Q(q^{-1},q)\right)\notag\\
&\label{t5e30}=\frac{q}{(1-q)} \left(-\sum_{n=0}^\infty ge(n)q^n+\frac{1}{(q; q)_\infty}+\frac{1}{(q; q)_\infty}\sum_{n=1}^{\infty}\dfrac{q^n(q^n;q)_n}{(q;q)_n}\right).
\end{align}
From  \cite{1,7}, we note that
\begin{equation}\label{bf}
    (\alpha;q)_\infty=(\alpha;q)_n(\alpha q^n;q)_\infty
\end{equation}
Using \eqref{bf} in  \eqref{t5e30}, we obtain 
\begin{align}
\sum_{n=1}^\infty K_{<}(n) q^n
&=\label{T5e3}\frac{q}{(1-q)} \left(-\sum_{n=0}^\infty ge(n)q^n+\frac{1}{(q; q)_\infty}+\frac{1}{(q; q)^2_\infty}\sum_{n=1}^{\infty}q^n(q^{n+1};q)_\infty(q^n;q)_n\right).
\end{align}
Multiplying both sides of \eqref{T5e3} by $(q; q)_\infty$ and then emplyoing  \eqref{e1},  \eqref{f} and \eqref{F}, we obtain 
$$\hspace{-7cm}
 \left(1-q\right)\left(\sum_{m=-\infty}^{\infty} (-1)^m q^{m(3m-1)/2}\right)\left(\sum_{n=1}^\infty K_{<}(n) q^n\right)$$
 \begin{equation}\label{T5e4}= \left(\sum_{m=-\infty}^{\infty} (-1)^m q^{m(3m-1)/2}\right)\left(-\sum_{n=0}^\infty ge(n)q^{n+1}\right)+q+\left(\sum_{m=0}^{\infty}p(m)q^m\right)\left(\sum_{n=1}^{\infty}F'(n)q^{n+1}\right).   \end{equation}
Comparing the like powers of $q$ on both sides of \eqref{T5e4} for $n\geq2$, we obtain
\begin{align}
&\hspace{-1cm}\sum_{m=\lfloor\frac{1-\sqrt{1+24n}}{6}\rfloor+1}^{\lceil\frac{1+\sqrt{1+24n}}{6}\rceil-1}(-1)^mK_{<}\left(n-P_m\right)-\sum_{m=\lfloor\frac{1-\sqrt{1+24(n-1)}}{6}\rfloor+1}^{\lceil\frac{1+\sqrt{1+24(n-1)}}{6}\rceil-1}(-1)^mK_{<}(n-P_m-1)\notag\\
&\label{T5e5}=\sum_{m=\lceil\frac{1-\sqrt{1+24(n-1)}}{6}\rceil}^{\lfloor\frac{1+\sqrt{1+24(n-1)}}{6}\rfloor}(-1)^{m+1}ge(n-P_m-1)+\sum_{m=0}^{n-2}p(m)F'(n-m-1).  
\end{align}
Hence, the proof is complete.
\end{proof}

\begin{theorem}\label{t4} For integers $n>1$ and $\alpha>0$,  we have
    $$\hspace{-10.5cm}(i)~
         K_{\leq}(3^\alpha n)\equiv K_{\leq}(n)\pmod 2.
    $$ 
    $$\hspace{-7.9cm}(ii)~
        3\sum_{n=1}^{\infty}K_{\leq}(3n+1)q^n\equiv \phi(q)f(q;q^5)\pmod 2.
    $$
    $$\hspace{-11.0cm}(iii)~
         K_{\leq}(3n+2)\equiv 0\pmod2.
    $$
    $$\hspace{-11.0cm}(iv)~
         K_{\leq}(5n+2)\equiv 0\pmod2.
    $$
    $$\hspace{-9.6cm}(v)~
         K_{>}(3n+2)\equiv p(3n+2)\pmod2.
    $$
    $$\hspace{-9.6cm}(vi)~
         K_{>}(5n+2)\equiv p(5n+2)\pmod2.
    $$
\end{theorem}
\begin{proof}
From \cite[Theorem 3.7.9]{bc}, we have
\begin{equation}\label{t3.7.9}
\phi(q)\phi(q^3)+4q\psi(q^2)\psi(q^6)=1+6\sum_{k=0}^{\infty}\left(\dfrac{q^{3k+1}}{1-q^{3k+1}}-\dfrac{q^{3k+2}}{1-q^{3k+2}}\right).
\end{equation}
Employing \eqref{t3.7.9} in \eqref{t1e11}, we obtain
\begin{equation}\label{new1}
    6\sum_{n=1}^{\infty}K_{\leq}(n)q^n\equiv\phi(q)\phi(q^3)-1\pmod4.
\end{equation}
Employing \eqref{newpsi1} in \eqref{new1}, we obtain
\begin{equation}\label{new2}
    6\sum_{n=1}^{\infty}K_{\leq}(n)q^n\equiv\phi(q^3)\left(\phi(q^9)+2qf(q^3,q^{15})\right)-1\pmod4.
\end{equation}
Extracting the terms involving $q^{3n}$ from both sides of \eqref{new2} and replacing $q^3$ by $q$, we obtain 
\begin{equation}\label{new4}
    6\sum_{n=1}^{\infty}K_{\leq}(3n)q^{n}\equiv\phi(q)\phi(q^3)\pmod4.
\end{equation}
Employing \eqref{new1} and \eqref{new4}, we obtain
\begin{equation}\label{new5}
6\sum_{n=1}^{\infty}K_{\leq}(3n)q^{n}\equiv6\sum_{n=1}^{\infty}K_{\leq}(n)q^{n}+1\pmod4.
\end{equation}
Extracting the coefficients of  $q^n$  from both sides of \eqref{new5}, we obtain
\begin{equation}\label{new6}
    6K_{\leq}(n)\equiv
    6K_{\leq}(3n)\pmod4.
\end{equation}
which implies
\begin{equation}\label{ty}
   K_{\leq}(n)\equiv
    K_{\leq}(3n)\pmod2.
\end{equation}
Iterating \eqref{ty} by replacing $n$ by $3n$, we arrive at (i).

Extracting the terms involving $q^{3n+1}$ from both sides of \eqref{new2},  dividing by $q$ and replacing $q^3$ by $q$, we arrive at (ii). 

Extracting the terms involving $q^{3n+2}$ from both sides of \eqref{new2},  dividing  by $q^2$ and replacing $q^3$ by $q$, we obtain 
\begin{equation}\label{new3}
     6\sum_{n=1}^{\infty}K_{\leq}(3n+2)q^n\equiv0\pmod4.
\end{equation}
Now (iii) easily follows from  \eqref{new3}. 

Employing \eqref{newpsi2} in \eqref{new1}, we obtain
\begin{equation}\label{new9}
     6\sum_{n=1}^{\infty}K_{\leq}(n)q^{n}\equiv\phi(q^3)\left(\phi(q^{25})+2qf(q^{15},q^{35})+2q^4f(q^5,q^{45})\right)\pmod4
\end{equation}Extracting the terms involving $q^{5n+2}$ from \eqref{new9}, dividing by $q^2$ and replacing $q^5$ by $q$, we arrive at (iv). 
 
 Employing \eqref{k=} and \eqref{k<} i result (iii) and (iv), we easily arrive ta (v) and (vi), respectively.
\end{proof}
\begin{theorem}For $m,$ $n,$ $k>0,$   $P_{\pm m}\equiv0\pmod5,$ and $T_{k}\equiv0\pmod5,$ we have
    \begin{align}
    \hspace{-4cm}(i)~& 
          K_{\leq}(5n+4+P_{\pm m})\equiv0\pmod{5m}. \notag\\
          \hspace{-4cm} (ii)~&
          K_{>}(5n+4+P_{\pm m+1})\equiv0\pmod{5m}. \notag\\
          \hspace{-4cm}  (iii)~&
          K_{<}(5n+3T_{k})\equiv K_{<}(5n+4+3T_{k})\pmod5.  \notag\\
          \hspace{-4cm}  (iv)~&
          K_{<}(5n+P_{-m})\equiv K_{<}(5n+4+P_{-m})\pmod5.  \notag
        \end{align}
\end{theorem}
\begin{proof} Employing Lemma \ref{RR5} in \eqref{k<=2}, we obtain
    \begin{align} 
   \hspace{-.2cm} \sum_{n=1}^\infty K_{\leq}(n) q^n=&\frac{(q^{25}; q^{25})_\infty^5}{(q^5; q^5)_\infty^6} (F^{-4}(q^5) + qF^{-3}(q^5)\notag\\
    & \hspace{-.5cm}+ 2q^2 F^{-2}(q^5) + 3q^3 F^{-1}(q^5) + 5q^4 - 3q^5 F(q^5)\notag\\ 
  & \hspace{-.5cm}\label{RR5e1}+ 2q^6 F^2(q^5) - q^7 F^3(q^5) + q^8 F^4(q^5))\left(\sum_{n=1}^\infty (-1)^{n-1} n q^{n(3n-1)/2} (1 + q^n)\right).
    \end{align}
 Extracting the terms involving $q^{5n+4+P_{\pm m}}$ from \eqref{RR5e1} such that $P_{\pm m}\equiv0\pmod5$, dividing both sides by $q^{4+P_{\pm m}},$ and then replacing $q^5$ by $q,$ we obtain
 \begin{equation}\label{m5E1}
         \sum_{n=1}^\infty K_{\leq}(5n+4+P_{\pm m})q^n=5\frac{(q^{5}; q^{5})_\infty^5}{(q; q)_\infty^6}\left((-1)^{m-1} m \right). 
        \end{equation}
 Now (i) follows easily from \eqref{m5E1}. (ii) follows from \eqref{k>2} and Lemma \ref{RR5}. Similarly, (iii) and (iv) follow from \eqref{k<2} and Lemma \ref{RR5}.
\end{proof}
\begin{theorem}For $m,$ $n,$ $k>0,$  $P_{\pm m}\equiv0\pmod7,$ and $T_{k}\equiv0\pmod7$, we have
  \begin{align} \hspace{-3.4cm}(i)&~ 
          K_{\leq}(7n+5+P_{\pm m})\equiv0\pmod7. 
      \notag\\
      \hspace{-3.4cm} (ii)&~
          K_{<}(7n+6+3T_{k})\equiv K_{<}(7n+5+3T_{k})\pmod7. 
        \notag\\
      \hspace{-3.4cm} (iii)&~
          K_{<}(7n+6+P_{-m})\equiv K_{<}(7n+5+P_{-m})\pmod7. \notag\\
        \hspace{-3.4cm}(iv)&~
          K_{>}(7n+5+P_{\pm m})\equiv0\pmod7. 
       \notag
       \end{align}
\end{theorem}
\begin{proof}Employing Lemma \ref{modp} in \eqref{k<=2}, we obtain 
    \begin{equation}\label{T7E1}
    \sum_{n=1}^\infty K_{\leq}(n) q^n\equiv\frac{(q;q)_\infty^6}{(q^7; q^7)_\infty} \sum_{n=1}^\infty (-1)^{n-1} n q^{n(3n-1)/2} (1 + q^n)\pmod7.
\end{equation}
    Employing Lemma \ref{m7} in \eqref{T7E1}, extracting the terms involving $q^{7n+5+P_{\pm m}}$ with $P_{\pm m}\equiv0\pmod7$, dividing both sides by $q^{5+P_{\pm m}}$ and then replacing $q^7$ by $q$, we obtain
    \begin{equation}\label{T7e1}
         \sum_{n=1}^\infty K_{\leq}(7n+5+P_{\pm m})q^n\equiv0\pmod7. 
        \end{equation} Now (i) follows immediately from \eqref{T7e1}. 
    Proofs of $(ii),$ $(iii),$ and $(iv)$ are identitcal to the proof of (i), so omitted.
\end{proof}

\begin{theorem}For $m,$ $n,$ $k>0,$  $P_{\pm m}\equiv0\pmod7,$ and $T_{k}\equiv0\pmod{11}$, we have
    \begin{align}
       \hspace{-3.3cm} (i)~&
          K_{\leq}(11n+6+P_{\pm m})\equiv0\pmod{11}\notag\\
       \hspace{-3.3cm} (ii)~&
          K_{<}(11n+7+3T_{k})\equiv K_{<}(11n+6+3T_{k})\pmod{11}.\notag\\
         \hspace{-3.3cm} (iii)~&
          K_{<}(11n+7+P_{-m})\equiv K_{<}(11n+6+P_{-m})\pmod{11}.
        \notag\\
        \hspace{-3.3cm}(iv)~&
          K_{>}(11n+6+P_{\pm m})\equiv0\pmod{11}.\notag
       \end{align}
\end{theorem}
\begin{proof} Employing Lemma \ref{modp} in \eqref{k<=2}, we obtain 
    \begin{equation}\label{m11E1}
    \sum_{n=1}^\infty K_{\leq}(n) q^n\equiv\dfrac{\left((q;q)_\infty^3\right)^7}{(q^{11}; q^{11})_\infty^2} \sum_{n=1}^\infty (-1)^{n-1} n q^{n(3n-1)/2} (1 + q^n)\pmod{11}.
\end{equation} 
Employing Lemma \ref{m11} in \eqref{m11E1}, we obtain 
\begin{equation}\label{m11E2}
    \dfrac{\left((q;q)_\infty^3\right)^7}{(q^{11}; q^{11})_\infty^2} \equiv\dfrac{\left(\mathcal{J}_0(q^{11})+q\mathcal{J}_1(q^{11})+q^3\mathcal{J}_3(q^{11})+q^6\mathcal{J}_6(q^{11})+q^{10}\mathcal{J}_{10}(q^{11})\right)^7}{(q^{11}; q^{11})_\infty^2}\pmod{11}.
\end{equation} 
Employing \eqref{m11E2} in \eqref{m11E1} and then extracting the terms involving $q^{11n+6+P_{\pm m}}$ where $P_{\pm m}\equiv0\pmod{11},$ we obtain
\begin{equation}\label{m11E3}
    \sum_{n=1}^\infty K_{\leq}(11n+6+P_{\pm m}) q^{11n+6+P_{\pm m}}\equiv\dfrac{Lq^{P_{\pm m}}}{(q^{11}; q^{11})_\infty^2}(-1)^{m-1}m\pmod{11}.
\end{equation} 
where from \cite{m11h}, 
\begin{align}
L=& 7q^{6}\mathcal{J}_0^6 \mathcal{J}_6 + 10q^{6}\mathcal{J}_0^5 \mathcal{J}_3^2 + q^{17}\mathcal{J}_0^4 \mathcal{J}_1 \mathcal{J}_6 \mathcal{J}_{10} + 8q^{6}\mathcal{J}_0^3 \mathcal{J}_1^3 \mathcal{J}_3 + 2q^{17}\mathcal{J}_0^3 \mathcal{J}_1 \mathcal{J}_3^2 \mathcal{J}_{10} + 8q^{28}\mathcal{J}_0^3 \mathcal{J}_6^3 \mathcal{J}_{10}\notag\\
&+ 3q^{17}\mathcal{J}_0^2 \mathcal{J}_1^2 \mathcal{J}_3 \mathcal{J}_6^2 + 3q^{28}\mathcal{J}_0^2 \mathcal{J}_1^2 \mathcal{J}_6 \mathcal{J}_{10}^2 + 3q^{28}\mathcal{J}_0^2 \mathcal{J}_3^2 \mathcal{J}_6^2 \mathcal{J}_{10} + 2q^{39}\mathcal{J}_0^2 \mathcal{J}_3 \mathcal{J}_6 \mathcal{J}_{10}^3 + 10q^{50}\mathcal{J}_0^2 \mathcal{J}_{10}^5 \notag\\
&+ 7q^{6}\mathcal{J}_0 \mathcal{J}_1^6 + q^{17}\mathcal{J}_0\mathcal{J}_1^4 \mathcal{J}_3 \mathcal{J}_{10} + 2q^{17}\mathcal{J}_0 \mathcal{J}_1^2 \mathcal{J}_3^3 \mathcal{J}_6 + 3q^{28}\mathcal{J}_0 \mathcal{J}_1^2 \mathcal{J}_3^2 \mathcal{J}_{10}^2 + q^{28}\mathcal{J}_0 \mathcal{J}_1 \mathcal{J}_3 \mathcal{J}_6^4 \notag\\
&+ 2q^{39}\mathcal{J}_0 \mathcal{J}_1 \mathcal{J}_6^3 \mathcal{J}_{10}^2 + q^{28}\mathcal{J}_0 \mathcal{J}_3^4 \mathcal{J}_6 \mathcal{J}_{10} + 8q^{39}\mathcal{J}_0 \mathcal{J}_3^3 \mathcal{J}_{10}^3 + 10q^{17}\mathcal{J}_1^5 \mathcal{J}_6^2 + 2q^{28}\mathcal{J}_1^3 \mathcal{J}_3 \mathcal{J}_6^2 \mathcal{J}_{10} \notag\\
&+ 8q^{39}\mathcal{J}_1^3 \mathcal{J}_6 \mathcal{J}_{10}^3 + 10q^{17}\mathcal{J}_1^2 \mathcal{J}_3^5 + 8q^{28}\mathcal{J}_1 \mathcal{J}_3^3 \mathcal{J}_6^3 + 3q^{39}\mathcal{J}_1 \mathcal{J}_3^2 \mathcal{J}_6^2 \mathcal{J}_{10}^2 +q^{50} \mathcal{J}_1 \mathcal{J}_3 \mathcal{J}_6 \mathcal{J}_{10}^4 \notag\\
&\label{md}+ 7q^{61}\mathcal{J}_1 \mathcal{J}_{10}^6 + 7q^{28}\mathcal{J}_3^6 \mathcal{J}_{10} + 7q^{39}\mathcal{J}_3 \mathcal{J}_6^6 + 10q^{50}\mathcal{J}_6^5 \mathcal{J}_{10}^2\notag\\
&\equiv0\pmod{11}.
\end{align}
Now (i) follows easily from \eqref{m11E3} and the fact that $L\equiv0\pmod {11}$. 
 Proofs of $(ii),$ $(iii),$ and $(iv)$ are similar to the  proof of $(i)$, so omitted.
\end{proof}
\section*{\bf Declarations}

\noindent\textbf{Funding}: This research did not receive funding.\\
\noindent{\bf Author Contributions.} Both authors contributed equally to this work.

\noindent{\bf Conflict of Interest.} The authors declare that there is no conflict of interest regarding the publication of this paper.

\noindent{\bf Human and Animal Rights.} The authors declare that there is no research involving human participants or animals in the context of this paper.	

\noindent{\bf Data Availability Statement.} Data sharing is not applicable to this paper as no datasets were generated or analyzed during the current study.	

\bibliographystyle{plain}

\end{document}